\documentclass[30pt]{amsart}
\newtheorem{theorem}{Theorem}[section]
\newtheorem{proposition}[theorem]{Proposition}
\newtheorem{lemma}[theorem]{Lemma}

\theoremstyle{definition}
\newtheorem{definition}[theorem]{Definition}

\theoremstyle{remark}
\newtheorem{remark}[theorem]{Remark}

\numberwithin{equation}{section}

\newcommand{\be}{\begin{equation}}
\newcommand{\ee}{\end{equation}}

\newcommand{\R}{\mathbb R}

\title[affine smoothed Wigner]{Smoothed affine Wigner transform}
\author[A. Athanassoulis T. Paul ]{A. Athanassoulis and T. Paul}
\address{CMLS \'Ecole polytechnique, 91 128 Palaiseau cedex}
\email{atah@math.polytechnique.fr}
\address{CNRS and CMLS \'Ecole polytechnique, 91 128 Palaiseau cedex}
\email{paul@math.polytechnique.fr}

\begin{document}
\hfill \textit{\`a la m\'emoire de Jean Morlet}
\vskip 3cm
\date{}
\maketitle

\tableofcontents

\begin{abstract}We study a generalization of Husimi function in the context of wavelets. This leads to a nonnegative  density on phase-space for which we compute the
evolution equation corresponding to a Schr\"odinger equation.
\end{abstract}


\section{Introduction}
Wigner and Husimi functions are important tools both in quantum mechanics and signal analysis. They both provide a
phase-space description of wave functions and signals.

Wigner provides a formulation of quantum mechnaics that seems , at first sight, very close to  transport equations for dynamics of
assembly of classical particles. However an important and inconvenient feature of Wigner functions is the fact that they are not pointwise nonnegative. This important difference with
classical densities causes  not only a problem of interpretation, it also creates technical problems.

The Husimi function is usually seen as the result of an appropriate smoothing the Wigner function, and this smoothing is enough to make it
positive. The simplest way of seeing this is to observe that the Husimi function is nothing but the square modulus of the
scalar product of the original function with a ``gaboret", namely a (Gaussian) coherent state. 

In the case where the original function satisfies a partial differential equation, e.g. a Schr\"odinger equation, 
Wigner and usual Schrd\"odinger representation s have been known to be, modulo a global phase, equivalent. 
Given a Schr\"odinger equation, there is a corresponding Wigner equation (see e.g. \cite{LP}) which is in general a pseudodifferential equation and 
whose limit, as the Planck constant $\hbar$ vanishes, is the classical Liouville equation. Although Wigner and Liouville equations are ``close" for small values
of $\hbar$, the non-positivity of the Wigner function forbids to  interpret it as a classical density, and therefore the Wigner equation
 as a perturbation of the Liouville one inside the paradigm of classical mechanics. It seems therefore natural to somehow try to overcome this difficulty by
describing the quantum system  in the  Husimi picture. However, by doing that, we encounter immediately the difficulty that 
the benefit of positivity is balanced by a  dynamics less easy to derive. This is due to the fact that  smoothing  is hard to ``effectively" invert.


Recently, \cite{AMP}, the equation that is satisfied by the Husimi function of a function, solution of a Schr\"odinger  equation, has been derived. This ``Husimi equation"
 involves in general analytic continuation in the complex domain, and can be expressed without complexification of the arguments in case of 
 analytic Hamiltonians, for example when the original Schr\"odinger operator is 
 a differential operator with polynomial coefficients.

In this paper we will derive the corresponding equation for the affine Husimi function. 
After a review of the standard Weyl-Wigner situation and the main results of \cite{AMP} in Section \ref{usual}, 
we will define  define ``Weyl" quantization  in the 
``ax+b" group setting in Section \ref{weyl}, and show how, using continuous wavelet (affine coherent states) one can define 
positive Husimi functions in Section \ref{wav}. we will then derive the 
corresponding Husimi equation
and express our main theorem Theorem \ref{the} in Section \ref{result} and suggest possible generalizations in section \ref{gene}.

\section{The usual Weyl case}\label{usual}
In this section we briefly review the standard smoothed Wigner results (see \cite{AMP} for details).

The fundamental equation of quantum mechanics, namely the Schr\"odinger equation, is a partial differential equation involving a
self-adjoint operator on a Hilbert space, generating a unitary flow. Usually the Hilbert space is $L^2(\R^n)$ and the unknown is the
so-called wave-function. This fact makes quantum mechnaics different from classical one by two facts : the fundamental equation is a PDE
and there is no, a priori, reference to a phase-space. despite this it has been recognized since the eraly days of quantum thery that a
phase-space description of teh wave-function is possible and useful. This fact has to be put together by the time-frequency methods in
signal analysis. One of the best ways to illustrate this common feature is to introduce the so-called Wigner-Ville transform.

On $\mathbb R^n$ the Wigner function associated to an $L^2(\R^n)$  function $\psi$ is by definition:
\be\label{wigner}
W_\psi(x,\xi)=\int_{\R^n}e^{-2i\pi\xi x}\overline{\psi(x+\hbar\frac y 2)}\psi(x-\hbar\frac y 2)dy
\ee
and its smoothed version is given by
\be\label{swigner}
\widetilde W_\psi(x,\xi)=(\sqrt 2/\hbar\sigma_x\sigma_\xi)^n
\int e^{-\frac{(x-x')^2}{\sigma_x^2\hbar}-\frac{(\xi-\xi')^2}{\sigma_\xi^2\hbar}}
W_\psi(y,\xi')dx'd\xi'.
\ee
It is well know that if $\psi$ satisfies a Schr\"odinger type equation of the type:
\be\label{schro}
i\hbar\partial_t\psi=h(x,-i\hbar\nabla_x)\psi,
\ee
the Wigner transform satisfies the following equation:
\be\label{eqwigner}
\partial_tW_\psi=2\Re\left(ih(x-\frac i 2\hbar\nabla_\xi,\xi+\frac i 2\hbar\nabla_x)W_\psi\right)
\ee
In what precedes $h(x,-i\hbar\nabla_x)$ is a pseudodifferential operator given by the Weyl calculus:
\be\label{weyl}
h(x,-i\hbar\nabla_x)u(x)=\hbar^{-n}\int h((x+y)/2,\xi)e^{i(x-y)/\hbar}u(y)dy
\ee
and $h(x-\frac i 2\hbar\nabla_\xi,\xi+\frac i 2\hbar\nabla_x)$ is  pseudodifferential operator on $L^2(\R^{2n})$ obtained the same way.

In \cite{AMP} was proven the following theorem;
\begin{theorem}
Let $\psi$ satisfies (\ref{schro}) with (for simplicity) $h\in\mathcal S(\R^{2n})$ (the Schwartz class). Then $W_{\psi^t}$ satisfies the following:
\be\label{cxz}
\partial_t\widetilde W_{\psi^t}= 2\Re\left(\widetilde{\mathcal L}\widetilde W_{\psi^t}\right)
\ee
where
\[
\widetilde{\mathcal L} w(x,\xi):=\int \hat h(S,T)e^{2\pi i(Sx+T\xi)-\frac\pi 2\hbar(\sigma_x^2 S^2+\sigma_\xi^2 T^2)}w\left(x+\hbar\frac{T+i\sigma_x^2S}2,
\xi-\hbar\frac{S-i\sigma_\xi^2 T}2\right)dSdT.
\]
(the possibility of extending the argument of $w$ is clear from the fact that the smoothed Wigner function is analytic on the whole complex plane). 
Here $\hat h$ is the Fourier transform of $h$.
\end{theorem}
 There are several ways of ``formally" deriving (\ref{cxz}). The first one consists in smoothing and unsmoothing 
 the equation (\ref{eqwigner}) and apply Egorov theorem associated to heat equation (i.e. free evolution at complex time).
 
 The second one consists in deriving the formula from the smoothed product itself.

 Finally, in the case $\sigma_x\times\sigma_k=1$ (and only in this case) we can derive (\ref{cxz}) by using the decomposition of identity on coherent states formula. 
 This is the method we will use for the wavelet case.

Let us remark mention that, in the case where the symbol $h$ is analytic 
(more precisely is the Fourier transform of a compact supported function) one can write th evolution equation for $\widetilde W_\psi$ as:
\[
\partial_t \widetilde W_{\psi^t}=2\Re\left(\widetilde{\mathcal L}_a \widetilde W_{\psi^t}\right),
\]
where $\widetilde{\mathcal L}_a$ has Weyl symbol:
\[
L(x,\xi;X,K):=h\left(x-\hbar\frac{K-i\sigma_x^2X}2,\xi-\hbar\frac{X+i\sigma_\xi^2K}2\right).
\]

\section{The affine Weyl quantization}\label{weyl}
Inspired by \cite{alex} affine Weyl quantization has been introduced in  (\cite{th}). We set the following .
Let on $L^2(\R^+, dx)$, $a>0$
\[
U(a,b)\varphi(x):=a^{1/2}e^{-\frac{bx}\hbar}\varphi(ax)
\]
\[
U(a,b)U(a',b')=U(aa',ab'+b).
\]
Let $I\varphi(x):=\varphi(-x)$ In analogy with the Weyl case we define:
\[
V(a,b):=U(a,b)IU(a,b)^{-1},
\]
and the Weyl quantization of $W(a,b)$ the operator:
\begin{equation}\label{we}
W=\int_{\R^+\times \R}W(a,b)V(a,b)\frac{dadb}{2a^2\hbar}.
\end{equation}
With the condition that $W\in L^1(\R^+\times\R,\frac{dadb}{a^2})$ (\ref{we}) defines clearly a (non uniformly in $\hbar$)
bounded operator on $L^2(\R^+)$.
\begin{remark}
We will not enter here in boundness condition uniform in $\hbar$ (Calderon-Vaillancourt type results). We refer to \cite{AP}
for an extensive study.
\end{remark}

\begin{lemma}\label{kernel}
if $W$ is given by (\ref{we}), then its integral kernel is (note the analogy with the standard Weyl quantization):
\be\label{ker}
w(x,y)=W(\frac 1 {\sqrt{xy}},\widehat{y-x}),
\ee
where $W(a,\widehat c)$ is the $\hbar$-Fourier transform in the second variable:
\[
W(a,\widehat c)=\frac 1 {\sqrt{2\pi\hbar}}\int_\R W(a,b)e^{-i\frac{bc}\hbar}db     . 
\]

In particular $\int_{\R^+\times \R}V(a,b)\frac{dadb}{2a\hbar}=Identity$.

\end{lemma}
\begin{proof}
We first notice that $U(a,b)^{-1}=U(\frac 1 a,-\frac ba)$. therefore:
\[
V(a,b)\psi(x)=\psi(\frac 1 {a^2x})e^{i(\frac b{a^2x\hbar}-\frac {bx}\hbar)},
\]
from which we deduce that the integral kernel of $V(a,b)$ is $\delta(y-\frac 1 {a^2x})e^{i(y-x)b/\hbar}$ from which we get (\ref{ker}).
\end{proof}

An easy computation shows the following

\begin{lemma}\label{trace}
Let $C(a)$ the operator defined on $L^2(\R^+)$ by:
\[
C(a)\varphi(x):=\frac 1 2\left(ax+\frac1  {ax}\right)\varphi(x).
\]
We have:
\be\label{inv}
W(a,b)=\mbox{Tr}[WV(a,b)C(a)].
\ee
\end{lemma}

\begin{remark}
the extra factor $C(a)$, not present in the standard Weyl quantization, is due to the non-unimodularity of the ``$ax+b$"
group.
\end{remark}

We define the affine-Wigner function of a function $\psi\in L^2(\R^+)$ as the Weyl symbol of the orthogonal projector on $\psi$, 
divided by the Planck constant.
\begin{definition}\label{aw}
\[
W_\psi(a,b):=\frac 1 \hbar\mbox{Tr}[WV(a,b)C(a)]=\langle \psi,V(a,b)C(a)\psi\rangle.
\]
\end{definition}
We get immediately the following expression:
\[
W_\psi(a,b)=\frac 1 {2\hbar}\int_{\R^+}\left(ax+\frac1  {ax}\right)\overline{\psi(x)}e^{-ib(x-\frac 1 {a^2x})}\psi\left(\frac 1 {a^2x}\right)dx.
\] and
\[
\int W_\psi(a,b)\frac{dadb}{a^2}=\Vert\psi\Vert_{L^2(\R^+)}.
\]

A study per se of the semiclassical limit of $W_\psi(a,b)$ has not, at our knowledge, done in the spirit of weak convergence 
results for the usual case \cite{LP,GMMP} and will be done in \cite{AP}. In this paper we will study ``directly"its smoothed version since, 
as in the standard Weyl case there is no reason for which the affine-Wigner function should be positive.

\section{Affine Husimi}\label{wav}

Coherents states have a long history in quantum mechanics (they were introduced by Sch\"odinegr in 1926), and also have got a more recent
interest in signal analysis through the so-called Gaborets. The afiine setting ihas produced the wavelets, and in qm corerspond to a
situation where the undelying phase-psace if the upeer half plane.

We define now the set of coherent states (continuous wavelet):
\begin{definition}
\[
\varphi_{a,b}(x)=C(\hbar)a^{1/\hbar+\frac 1 2}x^{1/\hbar}e^{-(a-ib)x/\hbar}=U(a,b)\varphi_{1,0}(x).
\]
\end{definition}
where $C(\hbar)=\sqrt{\frac{2^{1+2/\hbar}}{\Gamma{(\frac 2 \hbar+1)}}}$ is a normalization coefficient insuring that
\[
\Vert \varphi_{a,b}\Vert_{L^2(\R^+)}.
\]
We first remark that $\varphi_{a,b}$ decreases exponentially at $+\infty$ and as an increasing power as $\hbar\to 0$ at the origin.
  Ir results from general ``wavelet"theory (and a straightforward computation) the following decomposition of identity:
  \be\label{iden}
  \int_{\R\times\R^+} \vert \varphi_{a,b}><\varphi_{a,b}\vert\frac{dadb}{a\hbar}=Identity
  \ee
  
  We define now the wavelet-Husimi function:
\begin{definition}\label{WH}
The Husimi function of an operator $W$ of integral kernel $W(x,y)$ on $L^(\R^+)$ is s
\[
\widetilde W(a,b):=\frac{\langle \varphi_{a,b},W \varphi_{a,b}\rangle}\hbar.
\]
where $<.,.>$ is the scalar product on $L^2(\R^+)$.
\end{definition}

Let us first remark that, thanks to (\ref{iden}) we get immediately:
\be
\int_{\R\times\R^+}\widetilde W\frac{dadb}{a}=Tr W.
\ee
In the case where $W$ is the orthogonal projector $W_\psi$ on a function $\psi\in L^2(\R^+)$ (pure state) we get that
\be
\widetilde W_\psi(a,b):=C(\hbar)^2\vert \int_{\R^+}x^{1/\hbar}e^{-(a+ib)x/\hbar}\psi(x)dx\vert^2,
\ee
and:
\be
\int_{\R\times\R^+}\widetilde{W_\psi}(a,b)\frac{dadb}{a}=\Vert\psi\Vert^2_{L^2(\R^2)}.
\ee
Let us remark finally that, by construction, $\widetilde W_\psi\geq 0$ and $\widetilde W \geq 0$ if $W\geq 0$ as an operator.

It is easy to compute the wavelet-Husimi function out of the affine-Wigner one defined earlier.
\begin{proposition}\label{convo}
Let us denote $w_M(s,b)$ the Mellin transform in $a$ of $w(a,b)$ (see the Appendix).
We have:
\be\label{pos}
\widetilde W(a,b)=\frac{2^{1+2/\hbar}}{\hbar\Gamma{(\frac 2 \hbar+1)}}\int_{(1/2+i\R)\times\R}
\left(\frac{a^2+(b-\xi)^2}\hbar\right)^{\frac s 2 -\frac 1 \hbar-1}
\Gamma{(s/2-1/\hbar-1)}^2w_M(s,\xi)dsd\xi
\ee
Therefore if $W\geq 0$ as an operator (for example if $W=\vert\psi><\psi\vert$ for some vector $\psi$) and if $W(a,b)$ has a weak-limit the resulting measure is positive.
\end{proposition}
\begin{proof}
We first write that
\[
\widetilde W(a,b)=\frac{\langle \varphi_{ab},W\varphi_{ab}\rangle}\hbar.
\]
Therefore, since the integral kernel of $W$ is $W(\frac 1 {\sqrt{xy}},\widehat {x-y})$ we get, using the inverse of the 
Mellin transform given in the Appendix,
\[
\widetilde W(a,b)=C(\hbar)^2\int_{(\R^+)^2}\int_{1/2-i\infty}^{1/2+i\infty}(xy)^{s/2+1/\hbar} e^{-(a+i(b-\xi))x-(a-i(b-\xi))y}
w_M(s, \xi)dxdyd\xi ds
\]
using the equality:
\[
\lambda^{-k}\Gamma(k)=\int_{\R^+}u^{k-1}e^{-uy}du
\]
we get the (\ref{pos}). It is easy to check that, as $\hbar\to 0$, the  r.h.s. of (\ref{pos}) tends to $w(a,b)$ if this latter has a limit $w_0$ as a measure.
Therefore, since $\widetilde W\geq 0$ we get that $w_0$ must be positive.

\end{proof}

\section{The result}\label{result}
We can now state the main result of this paper.
\begin{theorem}\label{main}
Let $H$ be an operator of affine-Weyl symbol $h$ in $L^1(\R^+\times\R,\frac{dadb}{a^2})$. 
Let $\psi^t$ be the solution of the Schr\"odinger equation;
\be\label{dinger}
i\hbar\partial_t\psi^t=H\psi^t.
\ee
Then the wavelet-Husimi transform of $\psi^t,\ \widetilde W_{\psi}^t$ satisfies the following equation:
\be
\partial_t\widetilde W_{\psi^t}(a,b)=\frac 2 {\hbar}\Im\left(\int
\Phi(a,\alpha;b,\beta)\widetilde W_{\psi^t}(a+\frac{\alpha-i\beta}2,b+\frac{\beta+i\alpha}2)d\alpha d\beta\right)
\ee
where
\begin{eqnarray}\nonumber
\Phi(a,\alpha;b,\beta)&=&
\frac{2^{1+2/\hbar}}{\hbar\Gamma{(\frac 2 \hbar+1)}}\int_{(1/2+i\R)\times\R}\left(\frac{(a+i(b-\xi)(a+2\alpha-i(b+2\beta-\xi)}\hbar\right)^{\frac s 2 -\frac 1 \hbar-1}
\times\nonumber\\
&&\sqrt{\frac{a+\alpha+i(b+\beta)}{a-ib}}\Gamma{(s/2-1/\hbar-1)}^2w_M(s,\xi)dsd\xi\nonumber
\end{eqnarray}
and $\Im z$ is the imaginary part of $z$.

The same property holds for a non-pure state $W$ satisfying
$i\hbar\partial_tW^t=[H,W^t]$  by replacing $\tilde\psi^t$ by $\tilde W^t$.
\end{theorem}
Theorem \ref{main} is actually a direct consequence of the Theorem \ref{the} below. Indeed we note first that the equation (\ref{dinger}) can be rewritten as:
\[
\partial_t\vert\psi^t><\psi^t\vert=\frac 1 {i\hbar}[H,\vert\psi^t><\psi^t\vert].\] Therefore
\begin{eqnarray}\nonumber
\partial_t\langle\varphi_{(a,b)},\vert\psi^t><\psi^t\vert\varphi_{(a,b)}\rangle&=&
\langle\varphi_{(a,b)},\frac 1 {i\hbar}[H,\vert\psi^t><\psi^t\vert]\varphi_{(a,b)}\rangle\nonumber\\
&=& \frac 1 {i\hbar}\left(\langle\varphi_{(a,b)}H\vert\psi^t><\psi^t\vert\varphi_{(a,b)}\rangle-
\langle\varphi_{(a,b)}\vert\psi^t><\psi^t\vert H\varphi_{(a,b)}\rangle\right)\nonumber\\
&=&\frac 2 {\hbar}\Im\left(\langle\varphi_{(a,b)}H\vert\psi^t><\psi^t\vert\varphi_{(a,b)}\rangle\right)\nonumber
\end{eqnarray}

\begin{theorem}\label{the}
\be\label{theeq}
\widetilde W_\psi(a,b)=\int
\Phi(a,\alpha;b,\beta)\tilde\rho(a+\frac{\alpha-i\beta}2,b+\frac{\beta+i\alpha}2)d\alpha d\beta,
\ee
where
\begin{eqnarray}\nonumber
\Phi(a,\alpha;b,\beta)&=&
\frac{2^{1+2/\hbar}}{\hbar\Gamma{(\frac 2 \hbar+1)}}\int_{(1/2+i\R)\times\R}\left(\frac{(a+i(b-\xi)(a+2\alpha-i(b+2\beta-\xi)}\hbar\right)^{\frac s 2 -\frac 1 \hbar-1}
\times\nonumber\\
&&\sqrt{\frac{a+\alpha+i(b+\beta)}{a-ib}}\Gamma{(s/2-1/\hbar-1)}^2w_M(s,\xi)dsd\xi\nonumber
\end{eqnarray}

where $w_M(s,b)$ is the Mellin transform in the first variable.
\end{theorem}

\begin{proof}
the proof will use as one of the main ingredients the following Lemma.
\begin{lemma}
The wavelet-Husimi transform $\widetilde W_\psi(a,b)$ has an analytic continuation which is given by the following formula;
\be\label{analyt}
\widetilde W_\psi(a+\alpha+i\beta,b+\beta-i\alpha)=\sqrt{\frac{a+\alpha+i(b+\beta)}{a-ib}}\langle\varphi_{(a+2\alpha,b+2\beta)},W\varphi_{a,b)}\rangle
\ee

\end{lemma}
\begin{proof}
\end{proof}

We now use the decomposition of the identity by affine coherent states (wavelets) that is:
\be\label{deco}
\int \vert \varphi_{(a,b)}>
<\varphi_{(a,b)}\vert\frac{dadb}{a\hbar}=Identity.
\ee
plugging (\ref{deco}) in:
\be
\langle \varphi_{(a,b)}, HW\varphi_{(a,b)}\rangle=\int \langle \varphi_{(a,b)}, HW\varphi_{(a+2\alpha,b+2\beta)}\rangle
\langle \varphi_{(a+2\alpha,b+2\beta)}, HW\varphi_{(a,b)}\rangle\frac{d\alpha d\beta}{\alpha^2}
\ee we get the result thanks to the following Lemma, easy extension of Proposition \ref{convo}
\begin{lemma}
\begin{eqnarray}\nonumber
\langle\varphi_{(a',b')},W\varphi_{(a,b)}\rangle&=&
\frac{2^{1+2/\hbar}}{\hbar\Gamma{(\frac 2 \hbar+1)}}\int_{(1/2+i\R)\times\R}\left(\frac{(a+i(b-\xi)(a'-i(b'-\xi)}\hbar\right)^{\frac s 2 -\frac 1 \hbar-1}
\times\nonumber\\
&&\Gamma{(s/2-1/\hbar-1)}^2w_M(s,\xi)dsd\xi\nonumber
\end{eqnarray}
\end{lemma}
\end{proof}

\section{Possible generalizations}\label{gene}

One of the interest of coherent states is the fact that, since they are purely local, they can be defined in situation where there 
is no corresponding Weyl calculus, in particular in case of equations on manifolds \cite{UP}. They even are the building blocks 
 of semicalssical methods in 
the case where the phase-space is not even of the form $T^*\mathcal M$ (cotangent bundle), and is for example, a K\"ahlerain manifold.

We believe that it would be possible to derive the Husimi calculus in theses situations, and that it should have useful 
applications to PDEs theory 
(see \cite{bb,gg} for applications of affine Wigner functions in signal analysis).

\begin{appendix}
\section{The Mellin transform}
In this short appendix we recall some very well known facts about the Mellin  transform.

To a function $w$ defined on the positive axis  we associate its Mellin transform $w_M$ defined through:
\[
f_M(s):=\int_{\R^+}x^sw9x)\frac{dx}x.
\]
The transform $w(x)\to (2\pi)^{-1/2}w(1/2+it)$ is a unitary transform between $L^2(\R^+)$ and $L^2(\R)$ and, therefore,  the Mellin transform can be inverted thanks to the formula:
\[
w(x)=\frac 1 {2\pi i}\int_{1/2-i\infty}^{1/2+i\infty}x^{-s}w_M(s)ds.
\]
Let us note the link between Mellin transform and Laplace $L$ and Fourier $F$transforms:
\[
Fw(\xi)=\left(w\circ(-\log)\right)_M(i\xi),
\]
\[
Lw(\eta)=\left(w\circ(-\log)\right)_M(\eta).
\]

\end{appendix}
\vskip 1cm
\bf{Acknowledgments}\rm\ We would like to thanks the organizers of the 
conference ``Transform\'ee en ondelettes continues et ondelettes de Morlet
1978-2008" and especially Ginette Sarraco for her patience.

\end{document}